\documentclass[11pt]{amsart}%
\usepackage[utf8]{inputenc}%
\usepackage[T1]{fontenc}%
\usepackage[margin=2cm]{geometry}%

\usepackage{amsfonts}%
\usepackage{amssymb} 
\usepackage{cite}%
\usepackage{mathtools} 
\usepackage{tikz}%
\usepackage{tikz-3dplot}%
\usepackage[colorlinks,cite color=red,pagebackref=true,pdftex]{hyperref}%

\newcommand\Defn[1]{\textbf{#1}}%
\newcommand\eps{\varepsilon}%
\renewcommand\emptyset{\varnothing}%
\newcommand\R{\mathbb{R}}%
\newcommand\inner[1]{\langle {#1} \rangle}%
\newcommand\defeq{\coloneqq}%
\newcommand\eqdef{\eqqcolon}%
\newcommand\dcup{\mathbin{\mathaccent\cdot\cup}}%

\DeclareMathOperator{\relint}{relint}%
\DeclareMathOperator{\interior}{int}%
\DeclareMathOperator{\aff}{aff}%
\DeclareMathOperator{\conv}{conv}%
\DeclareMathOperator{\vol}{vol}%

\DeclareFontFamily{U}{mathx}{\hyphenchar\font45}%
\DeclareFontShape{U}{mathx}{m}{n}{ <5> <6> <7> <8> <9> <10> <10.95>
  <12> <14.4> <17.28> <20.74> <24.88> mathx10 }{}%
\DeclareSymbolFont{mathx}{U}{mathx}{m}{n}%
\DeclareFontSubstitution{U}{mathx}{m}{n}%
\DeclareMathAccent{\widecheck}{0}{mathx}{"71}%

\makeatletter
\newcommand{\citecomment}[2][]{\citen{#2}#1\citevar}
\newcommand{\citeone}[1]{\citecomment{#1}}
\newcommand{\citetwo}[2][]{\citecomment[,~#1]{#2}}
\newcommand{\citevar}{\@ifnextchar\bgroup{;~\citeone}{\@ifnextchar[{;~\citetwo}{]}}}
\newcommand{\citefirst}{\@ifnextchar\bgroup{\citeone}{\@ifnextchar[{\citetwo}{]}}}
\newcommand{\cites}{[\citefirst}
\makeatother

\newtheorem{thm}{Theorem}[section]%
\newtheorem{cor}[thm]{Corollary}%
\newtheorem{lem}[thm]{Lemma}%
\newtheorem{prop}[thm]{Proposition}%

\theoremstyle{definition}%
\newtheorem{example}[thm]{Example}%

\title{Angle sums of simplicial polytopes}%

\author{Sebastian Manecke}%
\address{Institut für Mathematik, Goethe-Universität Frankfurt,
  Germany}%
\email{manecke@math.uni-frankfurt.de}%


\makeatletter
\@namedef{subjclassname@2020}{%
  \textup{2020} Mathematics Subject Classification}
\makeatother

\date{\today}%

\begin{document}
\begin{abstract}
    The interior angle vector ($\widehat{\alpha}$-vector) of a
    polytope is a metric analogue of the $f$-vector in which faces are
    weighted by their solid angle. For simplicial polytopes,
    Dehn-Sommerville-type relations on the $\widehat{\alpha}$-vector
    were introduced by Sommerville (1927) and H\"ohn (1953). Camenga
    (2006) defined the $\widehat{\gamma}$-vector, a linear
    transformation analogous to the $h$-vector and conjectured it to
    be non-negative. Using tools from geometric and algebraic
    combinatorics, we prove this conjecture and show that the
    $\widehat{\gamma}$-vector increases in the first half and is
    flawless. In contrast to the $h$-vector, we construct a
    six-dimensional polytope with non-unimodal
    $\widehat{\gamma}$-vector. More generally, all result remain valid
    when solid angles are replaced by simple and non-negative cone
    valuations.
\end{abstract}

\maketitle
\nocite{sommerville1927}

\newcommand\aInt{\widehat{\alpha}}%
\newcommand\aExt{\widecheck{\alpha}}%
\newcommand\gInt{\widehat{\gamma}}%
\newcommand\gExt{\widecheck{\gamma}}%
\newcommand\Cones{\mathcal{C}}%

\newcommand\Arr{\mathcal{H}}

\section{Introduction}\label{sec:intro}
For a simplicial $d$-polytope $P$, the entry $f_i(P)$ of the
$f$-vector $f(P) = (f_{-1}(P), f_0(P), f_1(P), \dots, f_{d-1}(P))$ is
the number of $i$-dimensional faces of $P$. The $f$-vector of a
simplicial polytope has been a subject of investigation culminating in
the $g$-theorem due to Billera-Lee \cite{BilleraLee}, Stanley
\cite{Stanley_gTheorem}, and McMullen \cite{McMullen-Simple}, which
completely determines the whole set of possible $f$-vectors. The
result is best expressed by a certain linear transformation of the
$f$-vector called the $h$-vector
$h(P) = (h_0(P), h_1(P), \dots, h_d(P))$, which can be concisely
defined as an identity of polynomials in a single variable $t$:
\[
    \sum_{i=0}^{d}f_{i-1}(P) \cdot (t-1)^{d-i} \ = \ \sum_{k=0}^{d}h_{k}(P)
    \cdot t^{d-k}\,.
\]
Note that one can go back and forth between $f$- and $h$-vector, as
one defines the other. In terms of the $h$-vector, the following
linear relations and inequalities hold:
\begin{thm}[$h$-vectors of simplicial polytopes]\label{thm:hvectors}
    The $h$-vector
    $(h_0, \dots, h_d)$ of a simplicial $d$-polytope satisfies:
    \begin{enumerate}
    \item Dehn-Sommerville relations:
        $h_i = h_{d-i}$ for $i = 0, \dots, d$.
    \item Non-negativity: $0 \leq h_i$ for $i = 0, \dots, d$
    \item Unimodality:
        $h_0\leq h_1 \leq \dots \leq h_m \geq h_{m+1} \geq
        \dots \geq h_d$ for $m \defeq \big\lfloor\frac{d}{2}\big\rfloor$,
    \end{enumerate}
\end{thm}

\newcommand{\nuInt}{\widehat{\nu}}%
In this article, we focus on the interior angle-vector
$\nuInt(P) = (\nuInt_0(P), \dots, \nuInt_{d-1}(P))$ of a simplicial
polytope. The $\nuInt$-vector is a semi-discrete analogue of the
$f$-vector and can be defined by the sum of solid angle measures of
faces of fixed dimension.

Let $P \subseteq \R^d$ be a $d$-polytope. For a face $F$ of $P$,
define $\nuInt(F, P)$ to be the (solid) angle of $F$ at $P$:
\begin{equation}\label{eq:nu_def}
    \nuInt(F, P) \ \defeq \ \lim_{\eps \to 0} \frac{\vol_d(P \cap
      B_{\eps}(q))}{\vol_d(B_{\eps}(q))}\,,
\end{equation}
where $\vol_d$ is the $d$-dimensional volume and $B_{\eps}(q)$ is the
ball of radius $\eps$ around a point $q$ in the relative interior of
$F$. This expression is independent of the choice of $q$ and only
depends on $F$. Furthermore it is compatible with the intuitive
planar notion of angle in Euclidean geometry, but with a normalization
such that a full circle has angle $1$. Now the $i$-th entry of the
interior angle vector is just defined as the sum of the interior
angles at all $i$-faces:
\[
    \nuInt_i(P) \ \defeq \ \sum_{\substack{F \subset P\\\dim F = i}} \nuInt(F, P)\,.
\]

In \cite{BMS}, the notion of interior angle vector was generalized to
simple and normalized cone valuations $\alpha$, called \Defn{cone
  angles}, see \cite{McMullenAngleSumRelations} for a related
notion. A cone valuation is a map $\alpha : \Cones^d \to \R$, subject
to certain relations and we will give a precise definition of cone
valuations and cone angles in given in Section~\ref{sec:prelim}. The
notation $\aInt_i(P)$ will be used to denote the entries of the
interior angle vector of $P$ with respect to $\alpha$.

Cone valuations were used in \cite{BMS}, to determine the linear
equations satisfied by the $\aInt$-vector of all polytopes for any
cone angle $\alpha$. In this note, we focus on simplicial polytopes
and prove that similar linear inequalities as in
Theorem~\ref{thm:hvectors} are true for their $\aInt$-vectors --
again, for any cone angle $\alpha$. As before, the expressions in
terms of the $\aInt$-vector will be more complicated, but using the
same linear transformation, we derive the $\gInt$-vector from the
$\aInt$-vector, as introduced by Camenga \cite{Camenga2006}. Being
dependent on both $\alpha$ and $P$, we will denote the $\gInt$-vector
as $\gInt(\alpha, P)$, but will omit the $\alpha$ when it is clear
from the context. The \Defn{$\gInt$-vector}
$\gInt(P) = \gInt(\alpha, P)$ of a simplicial $d$-polytope
$P \subseteq \R^d$ is the tuple
$(\gInt_0(P), \gInt_1(P), \dots, \gInt_d(P))$ such that the following
identity of polynomials holds:
\[
    \sum_{k=0}^{d}\gInt_{k}(P) \cdot t^{d-k} \ \defeq \
    \sum_{i=0}^{d}\aInt_{i-1}(P) \cdot (t-1)^{d-i}\,,
\]
where we set $\aInt_{-1} = 0$. It was shown for $\alpha = \nu$ by Höhn~\cite{hohn}, that
over the set of all simplicial polytopes, the linear relations
satisfied by all $\aInt$-vectors form a
$\lfloor\frac{d}{2}\rfloor$-dimensional linear subspace. By
introducing the $\gInt$-vector, Camenga rewrote this equalities
into the following short equations reminiscent of the Dehn-Sommerville
equations:
\begin{thm}[Höhn~\cite{hohn}, Camenga~\cite{Camenga2006}] Let $P$ be a
    simplicial polytope and $\gInt(P) = \gInt(\nu, P)$ be the
    $\gInt$-vector with respect to the standard cone angle $\nu$. Then
    for $i = 0, \dots, d$:
    \[
        \gInt_i(P) + \gInt_{d-i}(P) \ = \ h_i(P)\,.
    \]
\end{thm}
Furthermore, she conjectured that the entries of the $\gInt$-vector
are non-negative. In this paper we prove this conjecture and show the
following:
\begin{thm}\label{thm:main}
    Let $\alpha$ be a non-negative cone angle, $P$ be a simplicial
    $d$-polytope and $\gInt(P) = \gInt(\alpha, P)$. Then the
    $\gInt$-vector $\gInt(P)$ satisfies:
    \begin{enumerate}
    \item Dehn-Sommerville equations: $\gInt_i(P) + \gInt_{d-i}(-P) = h_i(P)$
        for $i = 0, \dots, d$.
    \item Non-negativity: $0 \leq \gInt_i(P)$ for $i = 0, \dots, d$.
    \item Non-decreasing in the first half:
        $0 = \gInt_0(P) \leq \gInt_1(P) \dots \leq \gInt_m(P)$ for
        $m \defeq \big\lceil \frac{d}{2} \big\rceil$.
    \item Flawless: $\gInt_i(P) \leq \gInt_{d-i}(P)$ for
        $0 \leq i \leq \lfloor \tfrac{d}{2} \rfloor$.
    \end{enumerate}
\end{thm}
Let us make a few remarks. First, the observant reader may have
noticed that our take on the Dehn-Sommerville equations has a minus
sign in the second summand. We say that $P$ is
\Defn{$\alpha$-symmetric}, if $\aInt_i(P) = \aInt_i(-P)$ for
$i = 0, \dots, d-1$, which in turn implies that
$\gInt_i(P) = \gInt_i(-P)$. While for $\nu$ every polytope $P$ is
$\alpha$-symmetric, for most pairs $\alpha$ and $P$ this symmetry does
not hold, and the minus sign in the statement is necessary in these
cases.

Second, while the $\gInt$-vector is non-decreasing in the first half,
it is not unimodal. We show that unimodality only holds in low
dimensions:
\begin{thm}\label{thm:when_unimodal}
    Let $\alpha$ be a non-negative cone angle, $P$ be a simplicial
    $d$-polytope and $\gInt(P) = \gInt(\alpha, P)$. If $d \leq 3$, the
    $\gInt$-vector $\gInt(P)$ is unimodal. It is unimodal for all
    $\alpha$-symmetric $d$-polytopes if $d \leq 5$.
\end{thm}
Furthermore we give counterexamples for dimensions $4$
and $6$, respectively. Finally, we investigate the $\gInt$-vector of the simplest
simplicial polytope:
\begin{thm}\label{thm:simplex}
    Let $\alpha$ be a non-negative cone angle, $\triangle_d$ be a
    $d$-simplex. Then its $\gInt$-vector
    $(\gInt_0, \gInt_1, \dots, \gInt_d) = \gInt(\alpha, \triangle_d)$
    is non-decreasing:
    \[
        0 \ = \ \gInt_0 \leq \gInt_1 \leq  \dots \leq  \gInt_d\,.
    \]
    The converse is almost true: If $\gInt(P) = \gInt(\alpha, P)$ is
    non-decreasing for a simplicial polytope $P$, then $P$ is either a
    simplex or a bipyramid. In the latter case,
    $\gInt(P) = (0, 1, 1, \dots, 1, 1)$.
\end{thm}
We refer to Section~\ref{sec:unimodal} for a more precise discussion
of both cases. As a corollary, we will see that only one case occurs
when we consider the standard cone angle $\nu$, and we get the
following nice result:
\begin{cor}
    Let $P$ be a simplicial $d$-polytope and
    $\gInt(P) = \gInt(\nu, P)$. Then $\gInt(P)$ is non-decreasing if
    and only if $P$ is a $d$-simplex.
\end{cor}

After an introduction to cone angles in the next section, we
go on to prove the first half of Theorem~\ref{thm:main} in
Section~\ref{sec:gInt}, where we will introduce a method for
transferring linear inequalities from the $h$-vector of certain
relative simplicial complexes to inequalities of the
$\gInt$-vector. In Section~\ref{sec:stanley_reisner}, we will use
methods from algebraic combinatorics to show the necessary
inequalities of the relative simplicial complexes leading to the
second half of Theorem~\ref{thm:main}. Combining the work of the
previous two sections, we will be able two show
Theorem~\ref{thm:when_unimodal} and Theorem~\ref{thm:simplex} in
Section~\ref{sec:unimodal}.

\textbf{Acknowledgments.} The author would like to thank Raman Sanyal
for insightful discussions.

\section{Cone angles} \label{sec:prelim}%
\newcommand{\sphere}{\mathbb{S}}%
We begin with the thorough introduction of cone angles and interior
angle vectors. We call a map from the set $\Cones^d$ of all convex
polyhedral cones in $\R^d$ with apex at the origin to $\R$ a
\Defn{valuation}, if it satisfies the \Defn{valuation property}:
\[
    \alpha(C \cup C') \ = \ \alpha(C) + \alpha(C') - \alpha(C \cap C')
\]
for all $C, C' \in \Cones^d$ such that $C \cup C' \in \Cones^d$. We
call such a valuation \Defn{simple}, if $\alpha(C) = 0$ for all cones
$C \in \Cones^d$ with $\dim C < d$ and \Defn{normalized}, if
$\alpha(\R^d) = 1$. If the valuation $\alpha$ has both properties, we
call $\alpha$ a \Defn{cone angle}, cf.\ \cite{BMS}. If $\alpha(C) \geq 0$ for all
$C \in \Cones^d$, we say that $\alpha$ is \Defn{non-negative}.

One can think of cone valuations and cone angles as a generalization
of measures on the sphere. If $\mu$ is a measure on (the Borel
$\sigma$-algebra of) the sphere $\sphere^{d-1}$, we can define a
valuation, which we will also denote with $\mu$ by:
\[
    \mu(C) \ \defeq \ \mu(C \cap \sphere^{d-1}).
\]
The valuation property follows from the stronger $\sigma$-additivity of
measures. The most important example of a cone angle from a measure on
the sphere is the \Defn{(normalized) spherical volume} or the
\Defn{standard cone angle} $\nu : \Cones^d \to \R$:
\[
    \nu(C) \ \defeq \  \frac{\vol_d(C \cap B_1(0))}{\vol_d(B_1(0))}\,.
\]

While the most natural examples of cone angles arise from measures, it
is worth mentioning that not all valuations are obtained through this
construction. As an example, let $q \in \sphere^d$. Then
\begin{equation}\label{eq:omega_def}
    \omega_q(C) \ \defeq \ \lim_{\eps \to 0} \frac{\vol_d(B_\eps(q) \cap
    C)}{\vol_d(B_\eps(q))}
\end{equation}
defines a cone angle. We note that $\omega_q(C) = 0$, if $q \notin C$,
$\omega_q(C) = 1$, if $q \in \interior C$ and
$\omega_q(C) = \frac{1}{2}$, if $q$ is in the interior of a facet of
$C$. Thus $\omega_q$ is not continuous and therefore not a measure.

When evaluating a valuation, it is often necessary to subdivide a cone
into smaller pieces. This can be done by the Principle of
Inclusion-Exclusion:
\begin{prop}[Principle of Inclusion-Exclusion (IE)]\label{prop:IE} Let $\alpha : \Cones^d \to \R$ be any
    valuation and let $C = \bigcup_{i = 1}^n C_i$ for
    $C, C_1, \dots, C_n \in \Cones^d$. Then
    \[
        \alpha(C) \ = \ \sum_{\emptyset \neq I \subseteq [n]}
        (-1)^{|I|+1} \, \alpha \Big(\bigcap_{i \in I} C_i \Big)\,,
    \]
    where the sum is over all non-empty subsets of
    $[n] \defeq \{1, 2, \dots, n\}$. If additionally $\alpha$ is
    simple and $\dim (C_i \cap C_j) < d$ for $i \neq j$, then
    \[
        \alpha(C) \ = \ \sum_{i = 1}^n \alpha(C_i)\,.
    \]
\end{prop}

\newcommand{\faces}{\mathcal{F}}%
\newcommand{\pfaces}{{\mathcal{F}_+}}%
Let $P$ be a $d$-polytope given as the finite intersection of $n$
half-spaces
$H_k^{\leq} \defeq \{x \in \R^d : \inner{a_k, x} \leq b_k\}$ for
$k = 1, \dots, n$, each with a facet-defining hyperplane
$H_k \defeq \{x \in \R^d : \inner{a_k, x} = b_k\}$.  Let $\faces(P)$
be the face lattice of $P$ and let
$\pfaces(P) \defeq \faces(P) \setminus \{ \emptyset \}$ be the set of
non-empty faces. We will write $\faces_i(P)$ for the set of all
$i$-faces of $P$. For $q \in P$, define the \Defn{tangent
  cone} $T_qP$ to be
\begin{align*}
  T_qP \ \defeq& \ \{ x \in \R^d : q + \eps x \in P \text{ for some  $\eps > 0$ small} \}\,.
\end{align*}
The following well-known observation shows that this expression depends only on the unique face $F$ containing $q$ in its relative interior:

\begin{prop} Let $q \in \relint F$, where $F$ is a non-empty face of
    $P$. Then
\[
  T_qP \ = \ \{x \in \R^d : \inner{a_k, x} \leq 0 \text{ for all
      $k = 1, \dots, n$ with $F \subseteq H_k$} \}\,.
\]
\end{prop}
We therefore set $T_FP \defeq T_qP$ for some $q \in \relint F$ and we
also set $T_{\emptyset}P = \{0\}$. Evaluating the cone angle
$\alpha : \Cones^d \to \R$ at the tangent $T_FP$ of a face $F$ in a
$d$-polytope $P$ gives the \Defn{interior angle of $P$ at $F$ with
  respect to $\alpha$}, which we will denote with a hat above
$\alpha$:
\[
    \aInt(F, P) \ \defeq \ \alpha(T_FP)\,.
\]
Summing these expressions over all $i$-faces of $P$ gives the entries
of the \Defn{$\aInt$-vector} $\aInt(P)$, i.e. for $0 \leq i \leq d-1$,
let
\[
    \aInt_i(P) \ \defeq \ \sum_{F \in \faces_i(P)} \aInt(F, P)\,.
\]

Note that for $\alpha = \nu$, this gives precisely the definition of
$\nuInt$ in \eqref{eq:nu_def}.  We also remark that there is a
dual notion of an exterior angle $\aExt(F, P)$ and of an exterior angle
vector $\aExt(P)$, which comes from the evaluation of $\alpha$ at the
\Defn{normal cone $N_FP$}:
\[
    N_FP \ \defeq \ \{c \in \R^d : \inner{c, x} \geq \inner{c, y} \text{
      for all $x \in F$, $y \in P$}\} + \aff(F - F)\,.
\]
This definition of normal cone differs from the usual definition by
the term $\aff(F - F)$, which forces $N_FP$ to always be a
full-dimensional cone, so that $\aExt(F, P) \defeq \alpha(N_FP)$ is
not always zero for $\dim F > 0$. While being important in the overall
theory of angles, see \cite{BMS, McM-angle}, we will be only concerned
with interior and not with external angles here.

Let $H'_k \defeq \{x \in \R^d : \inner{a_k, x} = 0\}$ be the facet
defining hyperplanes of $P$ moved to the origin. We associate to $P$
the hyperplane arrangement $\Arr(P)$:
\[
    \Arr(P) \ \defeq \ \{H_1', H_2', \dots, H_n'\}\,.
\]
\newcommand{\Reg}{\mathcal{R}}%
The closures of the connected components of
$\R^d \setminus \bigcup \Arr(P)$ are called the \Defn{regions} of
$\Arr(P)$ and we denote the set of all of them by $\Reg(P)$. The
following is an easy, but crucial observation for the rest of the paper:
\begin{lem}\label{lem:H0_decomp}
    Let $P$ be a $d$-polytope and $F \in \pfaces(P)$. Each tangent
    cone $T_FP$ of some face $F$ of $P$ is a union of some regions
    in $\Reg(P)$, i.e.
    \[
        T_FP \ = \ \bigcup_{\substack{R \in \Reg(P)\\R \subseteq T_FP}} R\,.
    \]
\end{lem}
\begin{proof}
    It is clear that the right hand side is contained in the left hand
    side. By definition, $T_FP$ is bounded by the facet-defining
    hyperplanes of $P$ moved to the origin, so the facet-defining
    hyperplanes of $T_FP$ are in $\Arr(P)$ and $T_FP$ is a union of
    regions of $\Arr(P)$.
\end{proof}

\begin{example}\label{ex:pentagon1}
As an illustration, the following images shows a pentagon with one of
its vertices $v$ and the arrangement $\Arr(P)$, where the tangent cone
of $v$ and its dissection into the regions is highlighted.
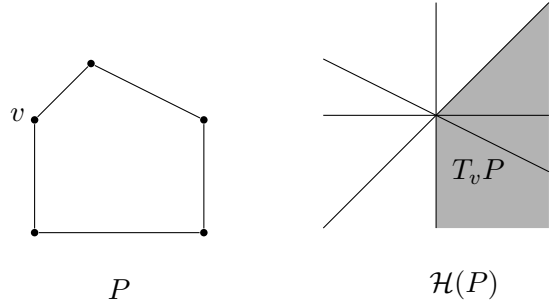
\begin{center}
    \begin{tikzpicture}[scale=0.75]
        \node[circle, inner sep=1pt, fill] (v1) at ( 2, 0) { };
        \node[circle, inner sep=1pt, fill] (v2) at ( 2, 2) { };
        \node[circle, inner sep=1pt, fill] (v3) at ( 0, 3) { };
        \node[circle, inner sep=1pt, fill] (v4) at (-1, 2) { };
        \node[circle, inner sep=1pt, fill] (v5) at (-1, 0) { };
        \node at (-1.3, 2.1) {$v$};
        \draw (v1) -- (v2) -- (v3) -- (v4) -- (v5) -- (v1);
        \node at (0.5, -1) {$P$};
    \end{tikzpicture}\qquad\qquad
    \begin{tikzpicture}[scale=0.75]
        \fill[gray!60!white] (0, 0) -- (2, 2) -- (2, -2) -- (0, -2);
        \draw ( 0,  -2) -- ( 0,  2);
        \draw (-2,   0) -- ( 2,  0);
        \draw (-2,   1) -- ( 2, -1);
        \draw (-2,  -2) -- ( 2,  2);
        \node at (0.75, -1) {$T_vP$};
        \node at (0.5, -3) {$\Arr(P)$};
    \end{tikzpicture}
\end{center}
\end{example}

Since the regions of a hyperplane arrangement intersect in lower
dimensional cones, we can, with the help of Proposition~\ref{prop:IE},
rewrite our quantity of interest, $\aInt_i(P)$, for
$0 \leq i \leq d-1$, as follows:
\begin{equation}
\label{eq:key_obs}
    \aInt_i(P)
    \ = \  \sum_{F \in \faces_i(P)} \alpha(T_FP)
    \ \stackrel{IE}{=} \ 
    \sum_{F \in \faces_i(P)} \sum_{\substack{R \in \Reg(P)\\R \subseteq T_FP}}\alpha(R)
    \ = \ \sum_{R \in \Reg(P)} \alpha(R) \sum_{\substack{F \in \faces_i(P)\\R \subseteq T_FP}} 1\,.
\end{equation}

Motivated by this expression, we find it useful to collect precisely
those $F$, for which $R \subseteq T_FP$ in a relative polytopal
complex. A \Defn{polytopal complex} $\Pi$ is a finite collection of
polytopes in $\R^d$, such that for any face $F$ of a polytope
$P \in \Pi$, we have $F \in \Pi$ and for any two polytopes
$P, P' \in \Pi$, their intersection $P \cap P'$ is a common face of
both $P$ and $P'$, see \cite[Chapter~5.1]{ziegler}. If
$\Gamma \subseteq \Pi$ is a subcomplex of $\Pi$, we call the pair
$\Psi = (\Pi, \Gamma)$ a \Defn{relative polytopal complex} and think
of $\Psi$ as $\Pi$ with the elements of $\Gamma$ being removed. Thus
we will write $P \in \Psi$, if $P \in \Pi$, but $P \notin \Gamma$. The
elements of $\Pi$ (respectively $\Psi$), are called their \Defn{faces}
and the maximal faces under inclusion are called \Defn{facets}. A
(relative) polytopal complex is a (relative) \Defn{$d$-complex}, if
the maximal dimension of one of its facet is $d$ and it is called
\Defn{pure}, if all its facets have the same dimension. Examples of
pure polytopal complexes are the set of all faces of a polytope $P$,
i.e. $\faces(P)$, as well as their boundary complexes
$\partial P \defeq \faces(P) \setminus \{P\}$.

Our main players will be (relative) polytopal complexes arising from
projections. Let $r \in \interior R$ for some region $R \in \Reg(P)$
where $P \subseteq \R^d$ is a $d$-polytope. If we imagine that $r$
points at a light source ``at infinity'', then $F$ is \emph{bright},
if and only if $r \notin T_FP$ and \emph{dark} otherwise. By
Lemma~\ref{lem:H0_decomp}, the set of facets which are bright/dark
does only depend on $R$ and not on the choice of $r$, thus let
$\bar{B}(R, P) \defeq \{F \in \partial P : R \not\subseteq T_FP\}$ be
the subcomplex of $\partial P$ formed by the bright faces and
$D(R, P) = (\partial P, \bar{B}(R, P))$ be the relative subcomplex
formed by the dark faces. Similarly, define
\begin{align*}
  \bar{D}(R, P) \ &\defeq \ \{F \in \partial P : -R \not\subseteq T_FP\} = \bar{B}(-R, P)\,,\\
  B(R, P) \ &\defeq \ (\partial P, \bar{D}(R, P)) = D(-R, P)\,.
\end{align*}
Topologically, $\bar{B}(R, P)$ is just the closure of $B(R, P)$ and the same for $\bar{D}(R, P)$. Finally, let us define the set of faces right at the boundary from bright to dark:
\begin{align*}
    \pi(R, P) \ &\defeq \ \{F \in \partial P : R, -R \not\subseteq T_FP\}\,.
\end{align*}
Clearly, $\pi(R, P)$ is isomorphic to $\partial P'$, where $P'$ is the
image of an orthogonal projection along a ray $r \in \interior R$.
\begin{prop} For all $d$-polytopes $P$ and $R \in \Reg(P)$, the sets
    $\bar{B}(R, P)$ and $\bar{D}(R, P)$ are pure $(d-1)$-complexes. The
    sets $B(R, P)$ and $D(R, P)$ are pure relative $(d-1)$-complexes,
    whereas the set $\pi(R, P)$ is a pure $(d-2)$-complex.
\end{prop}
\begin{proof} We only need to check that $\bar{B}(R, P)$ is a
    polytopal complex, the other statements follow immediately. Since
    $\bar{B}(R, P)$ is a subset of $\partial P$, we only need to show
    that for faces $F \subseteq G$ we have $G \in \bar{B}(R, P)$
    implies $F \in \bar{B}(R, P)$. But for $F \subseteq G$ we have
    $T_FP \subseteq T_GP$, so if $R \not\subseteq T_GP$, then
    $R \not\subseteq T_FP$.
\end{proof}

\begin{example}
    Continuing Example~\ref{ex:pentagon1}, for a region $R$ the
    following images depict the corresponding relative complexes
    $D(R, P)$, $\pi(R, P)$ and $B(R, P)$.
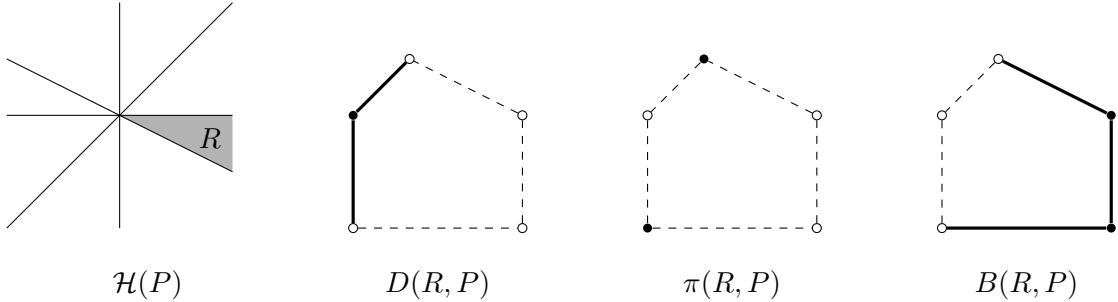
    \begin{center}
        \begin{tikzpicture}[scale=0.75]
            \fill[gray!60!white] (2, -1) -- (0, 0) -- (2, 0);
            \draw ( 0,  -2) -- ( 0,  2);
            \draw (-2,   0) -- ( 2,  0);
            \draw (-2,   1) -- ( 2, -1);
            \draw (-2,  -2) -- ( 2,  2);
            \node at (1.6, -0.4) {$R$};
            \node at (0.5, -3) {$\Arr(P)$};
        \end{tikzpicture}\qquad\qquad
        \begin{tikzpicture}[scale=0.75]
            \node[circle, inner sep=1.2pt, draw] (v1) at ( 2, 0) { };
            \node[circle, inner sep=1.2pt, draw] (v2) at ( 2, 2) { };
            \node[circle, inner sep=1.2pt, draw] (v3) at ( 0, 3) { };
            \node[circle, inner sep=1.2pt, fill] (v4) at (-1, 2) { };
            \node[circle, inner sep=1.2pt, draw] (v5) at (-1, 0) { };
            \draw[dashed] (v5) -- (v1) -- (v2) -- (v3);
            \draw[very thick] (v3) -- (v4) -- (v5);
            \node at (0.5, -1) {$D(R, P)$};
        \end{tikzpicture}\qquad\qquad
        \begin{tikzpicture}[scale=0.75]
            \node[circle, inner sep=1.2pt, draw] (v1) at ( 2, 0) { };
            \node[circle, inner sep=1.2pt, draw] (v2) at ( 2, 2) { };
            \node[circle, inner sep=1.2pt, fill] (v3) at ( 0, 3) { };
            \node[circle, inner sep=1.2pt, draw] (v4) at (-1, 2) { };
            \node[circle, inner sep=1.2pt, fill] (v5) at (-1, 0) { };
            \draw[dashed] (v1) -- (v2) -- (v3) -- (v4) -- (v5) -- (v1);
            \node at (0.5, -1) {$\pi(R, P)$};
        \end{tikzpicture}\qquad\qquad
        \begin{tikzpicture}[scale=0.75]
            \node[circle, inner sep=1.2pt, fill] (v1) at ( 2, 0) { };
            \node[circle, inner sep=1.2pt, fill] (v2) at ( 2, 2) { };
            \node[circle, inner sep=1.2pt, draw] (v3) at ( 0, 3) { };
            \node[circle, inner sep=1.2pt, draw] (v4) at (-1, 2) { };
            \node[circle, inner sep=1.2pt, draw] (v5) at (-1, 0) { };
            \draw[very thick] (v5) -- (v1) -- (v2) -- (v3);
            \draw[dashed] (v3) -- (v4) -- (v5);
            \node at (0.5, -1) {$B(R, P)$};
        \end{tikzpicture}\qquad\qquad
    \end{center}
\end{example}

The \Defn{$f$-vector} of a $d$-complex $\Pi$ is the vector
$(f_{-1}(\Pi), \dots, f_d(\Pi))$, where $f_i(\Pi)$ counts the number
of $i$-dimensional faces of $\Pi$. With this notation at hand, we can
rewrite the right hand side of \eqref{eq:key_obs} into the more concise form:

\[
    \aInt_i(P) \ = \ \sum_{R \in \Reg(P)} \alpha(R) f_i(D(R, P))\,.
\]

We think of $f$- and $\aInt$-vectors, as well as the $h$- and
$\gInt$-vectors as real-valued sequences, so that we can freely do
arithmetic with them. Any entry outside of the stated range is assumed
to be zero.

In \cite{PS67}, the following statement was proven in the case $\nu = \alpha$ although it is hinted that it holds more generally for even cone angles. 
\begin{thm}\label{thm:angle_proj}
    Let $\alpha$ be a non-negative cone angle and $P \subseteq \R^d$
    be a $d$-polytope. Then
    \[
        \aInt(P)+\aInt(-P) \ \in \ \conv\{f(\partial P) - f(\pi(R, P))
        : R \in \Reg(P)\}\,,
    \]
    where any entry of the $\aInt$-vector outside its usual range is set to zero.
\end{thm}
\begin{proof}
    Note that $\Arr(P) = \Arr(-P)$, so their regions coincide:
    $\Reg(P) = \Reg(-P)$.  By reflection in the origin, we can
    transform $D(R, -P)$ into $D(-R, P) = B(R, P)$, and we
    conclude $f(D(R, -P)) = f(B(R, P))$. Furthermore by the
    decomposition of the boundary of $P$,
    $\partial P = D(R, P) \dcup \pi(R, P) \dcup B(R, P)$, one also
    has $f(D(R, P)) + f(B(R, P)) = f(\partial P) - f(\pi(R, P))$.
    Therefore
    \begin{align*}
      \aInt(P) + \aInt(-P)
      \ =& \ \sum_{R \in \Reg(P)} f(D(R, P)) \cdot \alpha(R) +
         \sum_{R \in \Reg(-P)} f(D(R, -P)) \cdot \alpha(R)\\
      \ =& \ \sum_{R \in \Reg(P)} \big( f(D(R, P)) + f(B(R, P)) \big) \cdot \alpha(R)\\
      \ =& \ \sum_{R \in \Reg(P)} \big( f(\partial P) - f(\pi(R, P)) \big) \cdot \alpha(R).
    \end{align*}
    But this is the desired convex combination since $\sum_{R \in
\Reg(P)} \alpha(R) = 1$ and $\alpha(R) \geq 0$ for all $R \in \Reg(P)$.
\end{proof}

If $\alpha$ is an even cone angle, then, as we recall from the
introduction, $P$ is \Defn{$\alpha$-symmetric}, that is,
$\aInt(P) = \aInt(-P)$. In this case, the previous statement resolves
to:
\begin{equation}\label{eq:no_minus}
    2\aInt(P) \ \in \ \conv\{f(\partial P) - f(\pi(R, P))
    : R \in \Reg(P)\}\,.
\end{equation}
Besides the case that $\alpha$ is even, $\alpha$-symmetry does also
occur, if $P$ is \Defn{centrally symmetric}, i.e. if $P = -P$, in
which case $P$ is trivially $\alpha$-symmetric for all cone angles
$\alpha$. The next example shows that in general
$\aInt(P) \neq \aInt(-P)$ and that \eqref{eq:no_minus} does not
generalize to cone angles which are not $\alpha$-symmetric:
\begin{example}\label{ex:triangle} 
    Let $P = \conv(u, v, w)$ be any triangle in $\R^2$ with the origin
    in the interior and set $V = \{u, v, w\}$. We have
    $f(\partial P) = (1, 3, 3)$. $\Arr(P)$ is a hyperplane arrangement
    consisting of three lines through the origin and
    $\Reg(P) = \{T_{u}P, T_{v}P, T_{w}P, -T_{u}P, -T_{v}P, -T_{w}P\}$.
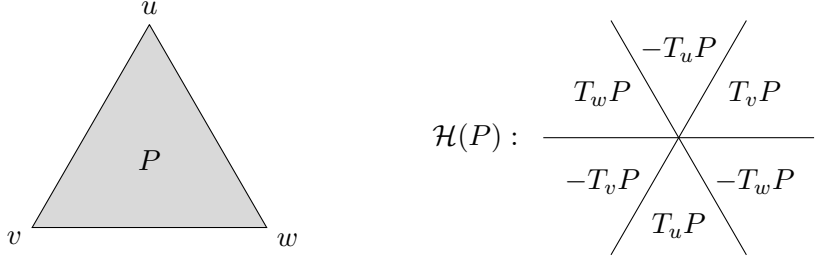
    \begin{center}
        \begin{tikzpicture}[scale=0.9]
            \draw[fill=gray!30!white] (90:2) -- (210:2) -- (330:2) --
            (90:2);%
            \node[anchor=-90] at ( 90:2) {$u$};
            \node[anchor= 30] at (210:2) {$v$};
            \node[anchor=150] at (330:2) {$w$};
            \node at (  0:0) {$P$};
        \end{tikzpicture}\qquad\qquad
        \begin{tikzpicture}[scale=0.9]
            \node at (-3, 0) {$\Arr(P):$};
            \draw (0:0) -- (  0:2);
            \draw (0:0) -- ( 60:2);
            \draw (0:0) -- (120:2);
            \draw (0:0) -- (180:2);
            \draw (0:0) -- (240:2);
            \draw (0:0) -- (300:2);
            \node at ( 30:1.3) {$T_vP$};
            \node at ( 90:1.3) {$-T_uP$};
            \node at (150:1.3) {$T_wP$};
            \node at (210:1.3) {$-T_vP$};
            \node at (270:1.3) {$T_uP$};
            \node at (330:1.3) {$-T_wP$};
        \end{tikzpicture}
    \end{center}
For $x \in V$, we have
    \begin{align*}
      f(D(T_xP, P)) \ = \ (0, 1, 2) \qquad \text{and} \qquad f(D(-T_xP, P)) \ = \ (0, 0, 1)\,.
    \end{align*}
    
    Since $0 \in \interior P$, we have $-x \in \interior T_xP$. Thus
    $\omega_{-x}(R) = 1$ if and only if $R = T_xP$ and $= 0$ otherwise
    for all regions $R \in \Reg(P)$. Let

    \[
        \alpha \ = \ \frac{1}{4}\sum_{x \in V} \omega_{-x} +
        \frac{1}{12} \sum_{x \in V} \omega_{x}\,.
    \]
    We calculate:
    \begin{align*}
      \aInt(P) &\ = \ \frac{1}{4}\sum_{x \in V} \omega_{-x}(T_x(P)) \cdot f(D(T_x(P), P)) +
                 \frac{1}{12} \sum_{x \in V} \omega_{x}(-T_xP) \cdot f(D(-T_x(P), P)) \\
               &\ = \ \frac{3}{4} (0, 1, 2) + \frac{3}{12} (0, 0, 1)
                 = \big(0, \tfrac{3}{2}, \tfrac{7}{4}\big)\,.
    \end{align*}
    Since the only combinatorial type of $\pi(R, P)$ is two points, we
    have $f(\pi(R, P)) = (1, 2, 0)$ for all $R \in \Reg(P)$, but
    \[
        2\aInt(P) \ = \ \big(0, 3, \tfrac{7}{2}\big) \ \neq \ (0, 1,
        3) \ = \ f(\partial P) - (1, 2, 0).
    \] 
\end{example}

The previous example also showed a construction of cone angles with
prescribed values on the regions of an arrangement $\Arr$. For future
reference, let us make this precise:
\begin{lem}\label{lem:alpha_construction}
    Let $\Arr$ be an arrangement of hyperplanes with regions
    $\{R_1, \dots, R_n\}$ and let
    $\lambda_1, \dots, \lambda_n \in [0, 1]$, such that
    $\sum_{i=1}^n \lambda_i = 1$. Then there exists a cone angle
    $\alpha$ such that $\alpha(R_i) = \lambda_i$.
\end{lem}
\begin{proof}
    Let $r_i \in \interior R_i$ be any ray. From the definition \eqref{eq:omega_def},
    we see
    \[
        \omega_{r_i}(R_j) \ = \ \begin{cases} 1& i = j\,,\\ 0&i \neq j\,. \end{cases}
    \]
    and therefore
    $\alpha \defeq \sum_{i = 1}^n \lambda_i \omega_{r_i}$ has the
    desired properties.
\end{proof}
\section{Simplicial polytopes and the $\gInt$-vector}\label{sec:gInt}
In this section, we will introduce the $\gInt$-vector and show its
non-negativity.

A (geometrical) \Defn{simplicial complex} $\Delta$ is a polytopal
complex where each face $\sigma \in \Delta$ is a simplex. Likewise, a
\Defn{relative simplicial complex} is a relative polytopal complex
$(\Delta, \Gamma)$, where $\Delta$ (and therefore $\Gamma$) is
simplicial. Two relative simplicial complexes $(\Delta, \Gamma)$ and
$(\Delta', \Gamma')$ are called \Defn{isomorphic},
$(\Delta, \Gamma) \cong (\Delta', \Gamma')$, if there exists a poset
isomorphism
$(\Delta \setminus \Gamma, \subseteq)
\stackrel{\tiny\cong}{\longrightarrow} (\Delta' \setminus \Gamma',
\subseteq)$.

The $f$-vectors of (relative) simplicial complexes have been a topic
of intensive study \cite{CKS_relative_simp_comp}. While unintuitive at
first, rewriting the $f$-vector into the $h$-vector has been a
fruitful approach. Recall from the introduction, that the
\Defn{$h$-vector} of a simplicial $(d-1)$-complex $\Delta$ is a linear
transformation of its $f$-vector, which can be compactly expressed as
an equation of polynomials:
\[ 
    \sum_{i=0}^{d}f_{i-1}(\Delta) \cdot (t-1)^{d-i} \ \eqdef \
    \sum_{k=0}^{d}h_{k}(\Delta) \cdot t^{d-k} \eqdef h(\Delta, t)\,.
\]
The right hand side is typically called the \Defn{$h$-polynomial} of
$\Delta$. As we want to think of $\Psi = (\Delta, \Gamma)$ as $\Delta$
with the faces in $\Gamma$ removed, we define the $f$- and $h$-vectors
of a relative simplicial $d$-complex $\Psi$ with
$\dim \Delta = \dim \Gamma = d-1$ as follows:
\[
    f_i(\Psi) \ \defeq \ f_i(\Delta) - f_i(\Gamma), \qquad h_i(\Psi) \
    \defeq \ h_i(\Delta) - h_i(\Gamma), \qquad h(\Psi, t) \defeq
    h(\Delta, t) - h(\Gamma, t)\,.
\]

In the same spirit, the \Defn{$\gInt$-vector}
$\gInt(\alpha, P) = (\gInt_0(P), \dots, \gInt_d(P))$ of a simplicial
$d$-polytope with respect to a cone angle $\alpha$ is defined by the
equality of the following polynomials:
\[ 
    \sum_{k=0}^{d}\gInt_{k}(P) \cdot t^{d-k} \ \defeq \
    \sum_{i=0}^{d}\aInt_{i-1}(P) \cdot (t-1)^{d-i}\,.
\]

We have seen that there is a connection between the $f$-vectors of the
polytopal complexes $D(R, P)$ for $R \in \Reg(P)$ and the
$\aInt$-vector of a polytope $P$. In the case that $P$ is simplicial,
$D(R, P)$ is a relative simplicial complex and this connection can be
transferred further from the $h$-vectors of $D(R, P)$ to the
$\gInt$-vector as defined in the introduction (see
Section~\ref{sec:intro}). More precisely:
\begin{align*}
  \sum_{k=0}^{d}\gInt_{k}(P) \cdot t^{d-k} \ \defeq& \ \sum_{i=0}^{d}\aInt_{i-1}(P) \cdot
  (t-1)^{d-i} \ = \ \sum_{i = 0}^d \sum_{R
                   \in \Reg(P)} f_{i-1}(D(R, P)) \cdot \alpha(R) \cdot (t-1)^{d-i}\\
  =& \ \sum_{R \in \Reg(P)} \alpha(R) \sum_{i = 0}^d f_{i-1}(D(R, P))
     \cdot (t-1)^{d-i}
  \ = \ \sum_{R \in \Reg(P)} \alpha(R) \sum_{k = 0}^d h_k(D(R, P)) t^{d-k}\,,
\end{align*}
or by comparing coefficients:
\begin{equation}\label{eq:gInt_from_regions}
    \gInt_k(P) \ = \ \sum_{R \in \Reg(P)} \alpha(R) \cdot h_k(D(R, P))\,.
\end{equation}

This insight allows us to rephrase (in-)equalities on the $h$-vector
of the $D(R, P)$'s into (in-)equalities on the $\gInt$-vector. Our
first application is a generalization of Theorem~\ref{thm:angle_proj}.
For its formulation we need to introduce the \Defn{$g$-vector}
$g(\Delta) = (g_0(\Delta), g_1(\Delta),\dots, g_d(\Delta))$ of a
simplicial $(d-2)$-complex $\Delta$. It is defined via:
\[
    g(\Delta, t) \ \defeq \ \sum_{k=0}^d g_k(\Delta) \cdot t^{d-k} \defeq (t-1)
    h(\Delta, t)\,.
\]

It immediately follows that $g_0(\Delta) = h_0(\Delta)$ and
$g_k(\Delta) = h_k(\Delta) - h_{k-1}(\Delta)$ for $1 \leq k \leq
d$. For the boundary complex of a simplicial $(d-1)$-polytope by the
Dehn-Sommerville relations one has
$g_i(\Delta) = h_i(\Delta) - h_{i-1}(\Delta) = h_{d-i}(\Delta) -
h_{d-i+1}(\Delta) = -g_{d-i}(\Delta)$ for $0 \leq i \leq \lfloor\frac{d}{2}\rfloor$
and for this reason typically the $g$-vector is defined to be only the
first half of what we wrote above. We use the extended version, which
makes it a bit easier to state our next proposition:
\begin{prop}\label{prop:gInt_angle_proj}
    Let $\alpha$ be any cone angle, $P$ be a simplicial $d$-polytope
    and $\gInt(P) = \gInt(\alpha, P)$. Then
    \[
        \gInt(P) + \gInt(-P) \ = \ h(\partial P) - \sum_{R \in
          \Reg(P)} \alpha(R) \cdot g(\pi(R, P))\,.
    \]
\end{prop}
\begin{proof}
    As in the proof of Theorem~\ref{thm:angle_proj}, we consider for
    each region $R \in \Reg(P)$ the decomposition of the boundary of
    $P$, $\partial P = D(R, P) \dcup \pi(R, P) \dcup B(R, P)$.  We
    conclude:
    \begin{align*}
      h(D(R, P), t) + h(B(R, P), t)
      \ &= \ \sum_{i=0}^{d}\big( f_{i-1}(D(R, P)) + f_{i-1}(B(R, P))\big)(t-1)^{d-i}\\
      \ &= \ \sum_{i=0}^{d}\big( f_{i-1}(\partial P) - f_{i-1}(\pi(R, P))\big)(t-1)^{d-i}\\
      \ &= \ h(\partial P, t) - (t-1) h(\pi(R, P), t) = h(\partial P, t) - g(\pi(R, P), t)\,.
    \end{align*}
    We highlight that we needed to use the $g$-vector here instead of
    the $h$-vector, since $\pi(R, P)$ is a $(d-2)$-complex, while
    $\partial P$, $D(R, P)$ and $B(R, P)$ are
    $(d-1)$-complexes. Noting that $B(R, P) = D(-R, P) \cong D(R, -P)$
    and summing over all regions $R$ gives the statement:
    \[
        \gInt(P) + \gInt(-P) \ = \ \sum_{R \in \Reg(P)} \alpha(R)\big(h(D(R, P)) +
        h(D(R, -P), t)\big) \ = \ h(\partial P) - \sum_{R \in \Reg(P)}
        \alpha(R) g(\pi(R, P))\,.\qedhere
      \]
\end{proof}

The remainder of this section is devoted to the first half of
Theorem~\ref{thm:main}. We will give an illustrative proof using
shellings, but the statement is also a byproduct of our more thorough
investigation of the algebraic combinatorics of these complexes in the
next section. We will provide the basic definitions and facts, for
proofs, we refer to \cite[Chapter~8]{ziegler}.

A \Defn{shelling} of a simplicial $(d-1)$-complex $\Delta$ is an ordering of
the facets $F_1, \dots, F_r$ of $\Delta$, such that for all
$1 \leq k \leq r$,
\[
    \Theta^k \ \defeq \ \{G \cap F_k \in \partial F_k \,:\, G \subseteq F_i \text{ for some } i =1,\dots,k - 1\}
\]
is a pure simplicial $(d-2)$-complex. If $\Delta$ has a shelling it is
called \Defn{shellable}. For a shellable simplicial complex $\Delta$,
let
\[
    \Delta^{k} \ \defeq \  \{G \in \Delta \,:\, G \subseteq F_i \text{ for some } i =1,\dots,k\}\,.
\]
An important property of shellings is their relation to the
$h$-vector. This can be most concretely seen when considering a single
shelling step from $\Delta^{k-1}$ to $\Delta^{k}$: Then
$h(\Delta^{k-1})$ and $h(\Delta^{k})$ are almost identical, except
that the $(d-i)$-th entry increased by one, where
$i \defeq f_{d-2}(\Theta^k)$ is the number of facets of $\Theta^k$.

A relative simplicial complex $\Psi = (\Delta, \Gamma)$ with pure
$\Delta$ and $\Gamma$ and $\dim \Delta = \dim \Gamma$ is called
\Defn{jointly shellable}, if there is a shelling $F_1, \dots, F_r$ of
the facets of $\Delta$, such that the facets $F_1, \dots, F_s$ of
$\Gamma$ are shelled first. Joint shellings are a special case of
(relative) shellings, see \cite[Chapter~III.7]{stanley} and
\cites{AS2016}, but are sufficient in our case.

Since the $h$-vector of a shellable simplicial complex increases in
each shelling step, we can conclude that jointly shellable simplicial
complex $\Psi = (\Delta, \Gamma)$ have
\begin{equation}
    \label{eq:relative_shellable_complexes_have_nonnegative_h_vector}
    h(\Psi) \ = \ h(\Delta) - h(\Gamma) \ \geq \ 0\,.
\end{equation}

An important class of shellable simplicial complexes are the boundary
complexes $\partial P$ of a simplicial polytopes $P$. In fact, there
is an easy way to find a shelling of $\partial P$ via line shellings:
Suppose that $0 \in \interior P$ and consider a ``generic'' line
$\ell$ through $0$. We begin at one of the intersection points of
$\ell$ with $\partial P$ and start to move on $\ell$ away from $P$. At
each point on $\ell$, we consider all the visible facets of $P$. Only
every so often a new facet will be show up, which was previously
hidden, and by taking $\ell$ generically enough this will happen one
facet at a time. Let us write down these facets in the order they will
occur: $F_1, \dots, F_s$. If we traveled $\ell$ long enough to reach
``infinity'', we will start our process again from the other side of
$\ell$ back to $P$. In this phase we will make a note whenever a
facets disappears and is not visible anymore: $F_{s+1}, \dots,
F_r$. Then the whole sequence $F_1, \dots, F_s, F_{s+1}, \dots F_s$ is a
shelling of $P$, called a \Defn{line shelling} of $P$. If we
start our tour on $\ell$ inside the cone $R$ of $\Reg(P)$, we can see
that $(\partial P)_s = \cup_{i=1}^s F_i$ is precisely $\bar{B}(R, P)$
and we have as a corollary:
\begin{cor}
    $D(R, P) = (\partial P, \bar{B}(R, P))$ is a jointly shellable
    simplicial complex.
\end{cor}

The following lemma is important for proving the Dehn-Sommerville
relations of the $\gInt$-vector. The statement is known as
Dehn-Sommerville relations for simplicial balls, see~\cite{chen_yan,
  novik_swartz} for a more general treatment of Dehn-Sommerville
relations for simplicial manifolds with boundary.
\begin{lem}\label{lem:h_i_to_h_d-i} $h_i(D(R, P)) = h_{d-i}(\bar{D}(R, P))$ and
    $h_i(B(R, P)) = h_{d-i}(\bar{B}(R, P))$ for all $0 \leq i \leq d$.
\end{lem}
\begin{proof}
    Since $B(R, P) = D(-R, P)$, we only need to show one of the
    equations. Consider a line shelling of $P$ giving a joint shelling
    $F_1, \dots, F_s, \dots, F_r$ of $\partial P$, which shells the
    facets $F_1, \dots, F_s$ of
    $\Gamma \defeq \bar{B}(R, P) = \Delta^s$ first. The reverse
    shelling $F_r, \dots, F_s, \dots, F_1$ is a shelling too, see
    \cite[Lemma~8.10]{ziegler}, so we define
    \begin{align*}
      \tilde\Theta^k \ &\defeq \ \{G \cap F_k \in \partial F_k \,:\, G \subseteq F_i \text{ for some } i = k + 1, \dots, d\}\\
      \tilde\Delta^{k} \ &\defeq \  \{G \in \Delta \,:\, G \subseteq F_i \text{ for some } i = k + 1, \dots, d\}\,.
      \end{align*}
    In $\Delta$ every ridge is contained in exactly two facets, and
    some of the facets sharing a ridge with $F_k$ are shelled before
    and some after $F_k$. The ones shelled after $F_k$ are precisely
    those which are shelled before $F_k$ in the reverse shelling,
    thus
    \[
        \Theta^k \cup \tilde\Theta^k \ = \ \partial F_k\,,
    \]
    and the contribution of $F_k$ to the $h$-vector is at the entry
    $d - f_{d-2}(\Theta^k)$ in the forward and at the entry
    $f_{d-2}(\Theta^k)$ in the reverse shelling. Considering the
    contributions of $F_k$ for $k > s$ both ways we obtain the second
    equality in:
    \[
        h_i(\Psi) \ = \ h_i(\Delta) - h_i(\Delta^s) \ = \
        h_{d-i}(\tilde\Delta^s) \ = \ h_{d-i}(\bar{D}(R, P))\,.\qedhere
    \]
\end{proof}

With this, we are able to prove the first half of Theorem~\ref{thm:main}:
\begin{prop}\label{prop:gInt_ds_non_negative}
    Let $P$ be a simplicial polytope and let $\gInt(P) = \gInt(\alpha, P)$,
    where $\alpha$ is a non-negative cone angle. The $\gInt$-vector
    $\gInt(P)$ satisfies
    \begin{enumerate}
    \item Dehn-Sommerville: $\gInt_i(P) + \gInt_{d-i}(-P) = h_i(P)$,
    \item Non-negativity: $\gInt_i(P) \geq 0$.
    \end{enumerate}
\end{prop}
\begin{proof}
    As outlined, we prove the properties by considering
    $h_i(D(R, P))$.  The non-negativity follows readily from
    \eqref{eq:gInt_from_regions} and the fact that both
    $\alpha(R) \geq 0$ and $\Psi \defeq D(R, P)$ satisfies
    \eqref{eq:relative_shellable_complexes_have_nonnegative_h_vector}.

    For the Dehn-Sommerville relations, we note that
    $D(R, -P) = B(-R, -P) \cong B(R, P)$ by reflection at the origin,
    and
    \begin{align*}
      \gInt_i(P) + \gInt_{d-i}(-P) \
      =& \ \sum_{R \in \Reg(P)} \alpha(R) \cdot \big(h_i(D(R, P)) + h_{d-i}(D(R, -P)) \big)\\
      =& \ \sum_{R \in \Reg(P)} \alpha(R) \cdot \big(h_i(D(R, P)) + h_{d-i}(B(R,  P)) \big)\\
      =& \ \sum_{R \in \Reg(P)} \alpha(R) \cdot \big(h_i(D(R, P)) + h_i(\bar{B}(R, P)) \big)
         = h_i(P)\,.\qedhere
    \end{align*}
\end{proof}
Remarkably, we can deduce that for $\alpha$-symmetric polytopes the
middle entry of the $\gInt$-vector is combinatorial in even
dimensions.

\begin{cor} Let $P$ be a simplicial and $\alpha$-symmetric
    $2m$-polytope. Then $2\gInt_m(P) = h_m(P)$.
\end{cor}
Again, we close with an example, in fact with the same example as
before. It shows, that the minus sign in our version of the
Dehn-Sommerville equations is necessary:
\begin{example}
    Let $P$ and $\alpha$ be as in Example~\ref{ex:triangle}. We have
    \begin{align*}
      h(D(T_xP, P)) \ = \ (0, 1, 1) \qquad \text{and} \qquad h(D(-T_xP, P)) \ = \ (0, 0, 1)
    \end{align*}
    and therefore $\gInt(P) = (0, \frac{3}{4}, 1)$. But since
    $h(P) = (1, 1, 1)$, the Dehn-Sommerville relations do not hold
    when neglecting the minus sign: $2\gInt_1(P) \neq h_1(P)$. On the
    other hand, $\gInt(-P) = (0, \frac{1}{4}, 1)$ and we see, that
    $\gInt_i(P) + \gInt_{2-i}(-P) = 1 = h_i$ as required.
\end{example}

\section{Relative Stanley-Reisner-Theory}\label{sec:stanley_reisner}
\renewcommand{\k}{\Bbbk}%
In this section we will review some of the concepts of relative
Stanley-Reisner-Theory and use it to prove inequalities on the
$h$-vectors of $D(R, P)$. These inequalities are already known, still,
for completeness we give proofs for them. We refer to
\cite{AS2016} and \cite[Chapter~III.7]{stanley} for details and further
applications.

Let $\Delta$ be a simplicial complex of dimension $d$ on $n$ vertices
and identify the vertices with $[n] \defeq \{1, 2, \dots, n\}$. For
each face $F \in \Delta$, let $V(F)$ be the subset of $[n]$, which
corresponds to the vertices of $F$. For an infinite field $\k$, let
$S \defeq \k[x_1, \dots, x_n]$ and let
$N(\Delta) \defeq 2^{[n]} \setminus \{V(F) : F \in \Delta\}$ be the
\Defn{set of non-faces} of $\Delta$. We define the
\Defn{Stanley-Reisner-Ideal} $I_\Delta \subseteq S$ of $\Delta$ as
follows:
\[
    I_\Delta \ \defeq \ \langle x^I : I \in N(\Delta) \rangle\,,
\]
where $x^I \defeq \prod_{i \in I} x^i$, and the
\Defn{Stanley-Reisner-Ring} $\k[\Delta] \defeq S / I_\Delta$. A
simplicial complex $\Delta$ is said to be a \Defn{Cohen-Macaulay
  complex} (over $\k$), if $\k[\Delta]$ is Cohen-Macaulay. For a
relative simplicial complex $\Psi = (\Delta, \Gamma)$, we similarly
define the \Defn{Stanley-Reisner-Module}
\[
    M[\Psi] \ \defeq \ {I_\Gamma} / {I_\Delta}\,.
\]

We can also think of $M[\Psi]$ as the kernel of the surjection
$\k[\Delta] \to \k[\Gamma]$, i.e.\ there exists a short exact sequence
of graded $S$-modules:
\begin{equation}\label{eq:exact_sequence}
    \begin{tikzpicture}
        \node (a0) at (-4.5, 0) {$0$};
        \node (a1) at (-2.5, 0) {$M[\Psi]$};
        \node (a2) at ( 0, 0) {$\k[\Delta]$};
        \node (a3) at ( 2.5, 0) {$\k[\Gamma]$};
        \node (a4) at ( 4.5, 0) {$0$\,.};
        \draw[->] (a0) -- (a1);
        \draw[->] (a1) -- (a2);
        \draw[->] (a2) -- (a3);
        \draw[->] (a3) -- (a4);
    \end{tikzpicture}
\end{equation}
Assume that both $\Delta$ and $\Gamma$ are $d$-dimensional
Cohen-Macaulay complexes. Since we assumed $\k$ to be infinite, by the
Kind-Kleinschmidt criterion \cite{kind_kleinschmidt} there exists a
regular sequence $(\theta_1, \dots, \theta_d)$ of $\k[\Delta]$, which
is also a regular sequence of $\k[\Gamma]$. We can recover the
$h$-vector of the $\Delta$ and $\Gamma$ by taking the quotient with
$\Theta \defeq \theta_1 S + \dots + \theta_d S$:
\[
    \dim {}(\k[\Delta]/\Theta)_i \ = \ h_i(\Delta)\,, \qquad
    \dim {}(\k[\Gamma]/\Theta)_i \ = \ h_i(\Gamma)\,.
\]

Then, from Lemma~1.1.4 in \cite{bruns_herzog} we
also have an exact sequence of graded $S$-modules when we quotient
\eqref{eq:exact_sequence} by $\Theta$:
\begin{center}
    \begin{tikzpicture}
        \node (a0) at (-4.5, 0) {$0$};
        \node (a1) at (-2.5, 0) {$M[\Psi] / \Theta$};
        \node (a2) at ( 0, 0) {$\k[\Delta] / \Theta$};
        \node (a3) at ( 2.5, 0) {$\k[\Gamma] / \Theta$};
        \node (a4) at ( 4.5, 0) {$0$\,.};
        \draw[->] (a0) -- (a1);
        \draw[->] (a1) -- (a2);
        \draw[->] (a2) -- (a3);
        \draw[->] (a3) -- (a4);
    \end{tikzpicture}
\end{center}

We immediately see that
$\dim {}(M[\Psi]/\Theta)_i = h_i(\Delta) - h_i(\Gamma) = h_i(\Psi) \geq 0$.
A linear form $\omega \in \k[\Delta]_1$, is called a \Defn{Lefschetz
  element} for $\k[\Delta]$, if
$(\k[\Delta]/\Theta)_i \xrightarrow{\cdot \omega^{d - 2i}}
(\k[\Delta]/\Theta)_{d-i}$ is an isomorphism. These linear forms
played an important role in Stanley's proof of the $g$-theorem
\cite{Stanley_gTheorem}, see also \cite{FlemmingKaru}. There, the
following is established:
\begin{thm}\label{thm:lefschetz_exists}
    Let $P$ be a simplicial $d$-polytope and let
    $(\theta_1, \dots, \theta_d)$ be a regular sequence of
    $\k[\partial P]$. Then $\k[\partial P] / \Theta$ has a Lefschetz
    element.
\end{thm}

We will be only concerned with $\Delta = \partial P$ for some
simplicial polytope $P$, so assume that there is in fact a Lefschetz
element $\omega$ of $\k[\Delta] / \Theta$ and let $i <
\frac{d}{2}$. Then the following diagram commutes:
\begin{center}
    \begin{tikzpicture}
        \node (a1) at (-3, 0) {$(M[\Psi] / \Theta)_i$};
        \node (a2) at ( 0, 0) {$(\k[\Delta] / \Theta)_i$};
        \node (a3) at (-3, -2) {$(M[\Psi] / \Theta)_{i+1}$};
        \node (a4) at ( 0, -2) {$(\k[\Delta] / \Theta)_{i+1}$};
        \draw[->] (a1) -- node[anchor=west]{$\cdot \omega$} (a3);
        \draw[->] (a1) -- (a2);
        \draw[->] (a2) -- node[anchor=west]{$\cdot \omega$} (a4);
        \draw[->] (a3) -- (a4);
    \end{tikzpicture}
\end{center}
and since all maps except for the left downwards arrow are
injective, it has to be injective as well. By looking at the
dimensions we see
\[
    h_i(\Psi) \ = \ \dim (M[\Psi]/\Theta)_i \ \leq \ \dim
    (M[\Psi]/\Theta)_{i+1} \ = \ h_{i+1}(\Psi).
\]

If we instead consider the isomorphism $\cdot \omega^{d - 2i}$, we get
the commutative diagram
\begin{center}
    \begin{tikzpicture}
        \node (a1) at (-3, 0) {$(M[\Psi] / \Theta)_i$};
        \node (a2) at ( 0, 0) {$(\k[\Delta] / \Theta)_i$};
        \node (a3) at (-3, -2) {$(M[\Psi] / \Theta)_{d-i}$};
        \node (a4) at ( 0, -2) {$(\k[\Delta] / \Theta)_{d-i}$};
        \draw[->] (a1) -- node[anchor=west]{$\cdot \omega^{d-2i}$} (a3);
        \draw[->] (a1) -- (a2);
        \draw[->] (a2) -- node[anchor=west]{$\cdot \omega^{d-2i}$} (a4);
        \draw[->] (a3) -- (a4);
    \end{tikzpicture}
\end{center}
and the same argument shows that the left downwards arrow is again
injective. Comparing the dimension gives us the second inequality:
\[
    h_i(\Psi) \ = \ \dim (M[\Psi]/\Theta)_i \ \leq \ \dim
    (M[\Psi]/\Theta)_{d-i} \ = \ h_{d-i}(\Psi).
\]

We have therefore shown:
\begin{prop}\label{prop:h_vector_nondecr_on_first_half}
    Let $\Psi = (\Delta, \Gamma)$ be a relative simplicial complex
    such that both $\Delta$ and $\Gamma$ are $d$-dimensional
    Cohen-Macaulay complexes and such that $\Delta$ has a Lefschetz
    element. Then $0 \leq h_{i-1}(\Psi) \leq h_{i}(\Psi)$ and
    $h_i(\Psi) \leq h_{d-i}(\Psi)$ for all
    $0 \leq i \leq m \defeq \lceil \tfrac{d}{2} \rceil$.
\end{prop}

As a corollary, we can prove the second half of Theorem~\ref{thm:main}:
\begin{cor}
    Let $P$ be a simplicial polytope and $\gInt(P) = \gInt(\alpha, P)$,
    where $\alpha$ is a non-negative cone angle.
    \begin{enumerate}
    \item $\gInt(P)$ increases in the first half:
        $\gInt_0(P) \leq \gInt_1(P) \leq \dots \leq \gInt_m(P)$, where
        $m \defeq \lceil \tfrac{d}{2} \rceil$, and
    \item $\gInt(P)$ is flawless: $\gInt_i(P) \leq \gInt_{d-i}(P)$ for
        all $0 \leq i \leq \lceil\frac{d}{2}\rceil$.
    \end{enumerate}
\end{cor}
\begin{proof}
    As we have seen, $D(R, P)$ is jointly shellable for each
    $R \in \Reg(P)$. Therefore, both $\Delta = \partial P$ and
    $\Gamma = \bar{B}(R, P)$ are Cohen-Macaulay complexes. By
    Theorem~\ref{thm:lefschetz_exists}, $\partial P$ has a Lefschetz
    element. Together with
    Proposition~\ref{prop:h_vector_nondecr_on_first_half} the
    statement is now an immediate consequence of
    \eqref{eq:gInt_from_regions}.
\end{proof}

\section{The $\gInt$-vector of a simplex and unimodality}\label{sec:unimodal}
Our inequalities give us a somewhat complete picture of the
$\gInt$-vector of a $d$-polytope $P$ in its first half. But besides
being flawless, we cannot say much about what happens in the second
half. In this section, we will look at the question of unimodality. A
sequence $(a_0, a_1, \dots, a_n)$ is called \Defn{unimodal}, if there
exists $0 \leq p \leq n$ such that
$a_0 \leq a_1 \leq \dots \leq a_p \geq \dots \geq a_n$. We show that
we can only hope for unimodality of the $\gInt$-vector in small
dimensions. Unless otherwise stated, let $P$ be an simplicial polytope
and $\gInt(P) = \gInt(\alpha, P)$ for some non-negative cone angle
$\alpha$.

From our previous results,
we have the following observation:
\begin{prop}
    Let $P$ be a simplicial $d$-polytope, $d \leq 3$. Then $\gInt(P)$
    is unimodal.
\end{prop}
\begin{proof}
    By Theorem~\ref{thm:main}~(3), the $\gInt$-vector of $P$ is
    non-decreasing until $\gInt_{d-1}$.
\end{proof}
Nonetheless, we can also see that unimodality fails for simplicial
polytopes in dimension $4$.
\begin{example}
    There exist $4$-dimensional simplicial polytopes with non-unimodal
    $\gInt$-vector, even if $\alpha$ is non-negative. For this take
    the cross-polytope $\Diamond_4$ and transform it projectively such
    that a orthogonal projection along some $u \in \R^4$ has a single
    simplex as image. Let $R \in \Reg(P)$ be the region containing
    $u$. By appealing to Lemma~\ref{lem:alpha_construction} we set
    $\alpha(R) = \frac{1}{6}$ and $\alpha(-R) = \frac{5}{6}$ and get
    \begin{align*}
      \gInt(\Diamond_4) \ &= \ \tfrac{5}{6} (0, 0, 0, 0, 1) + \tfrac{1}{6} (0, 4, 6, 4, 1)\\
       &= \ (0, \tfrac{2}{3}, 1, \tfrac{2}{3}, 1)\,.
    \end{align*}
\end{example}

More can be said, if $P$ is a simplex. We want to denote by
$\triangle_d$ any $d$-simplex in $\R^d$.

\begin{lem}\label{lem:simplex}
    $\gInt(\triangle_d)$ is non-decreasing.
\end{lem}
\begin{proof}
    Let $\Psi_k = (\triangle_d, \Gamma_k)$, where $\Gamma$ is formed by
    any subset of $k$ facets of $\Psi$, $1 \leq k \leq d$. Since any
    ordering of the facets is a shelling, we see that 
    \[
        h(\Gamma_k) \ = \ (\underbrace{1, 1, \dots, 1}_k, 0, 0, \dots, 0) 
    \]
    and therefore
    \[
        h(\Psi_k) \ = \ (\underbrace{0, 0, \dots, 0}_k, 1, 1, \dots, 1)\,.
    \]

    For every region $R \in \Reg(\triangle_d)$, we have
    $D(R, P) \cong \Psi_k$, where $k$ is the number of facets of
    $\bar{B}(R, K) \cong \Gamma_k$. Since the $h$-vector of $D(R, P)$
    is non-decreasing for any $R$, the $\gInt$-vector is too, by
    \eqref{eq:gInt_from_regions}.
\end{proof}

Interestingly, the converse is also almost true for $\alpha$-symmetric
polytopes. A bipyramid over a simplex, or simply a \Defn{bipyramid},
is any polytope combinatorially isomorphic to the convex hull of
$\triangle_{d-1}$ together with two points $v, v'$ on different sides
of the hyperplane $\aff(\triangle_{d-1})$ such that
$\conv(v, v') \cap \triangle_{d-1} \neq \emptyset$.
\begin{prop}
    Let $\alpha$ be a non-negative cone angle and let $P$ be a
    simplicial and $\alpha$-symmetric $d$-polytope, $d \geq 2$. Let
    $\gInt(P) \defeq \gInt(\alpha, P)$ be non-decreasing. Then $P$ is
    either a simplex or a bipyramid. In the latter case,
    $\gInt(P) = (0, 1, \dots, 1)$ and $\alpha(R) = 0$ for all regions
    $R\in \Reg(P)$ with
    $\pi(R, P) \not\cong \partial \triangle_{d-1}$.
\end{prop}
\begin{proof}
    Since $\gInt(P)$ is non-decreasing and $\gInt_d(P) = 1$, we have
    by the Dehn-Sommerville-relations,
    Proposition~\ref{prop:gInt_ds_non_negative}, for all
    $k \leq \frac{d}{2}$:
    \[
        h_k(\partial P) \ = \ \gInt_k(P) + \gInt_{d-k}(-P) \ = \
        \gInt_k(P) + \gInt_{d-k}(P) \ \leq \ 2\gInt_d(P) \ = \ 2\,.
    \]

    If $h_k(P) = 1$ for $k \geq 1$, then
    $1 \leq h_1(\partial P) \leq h_k(\partial P) = 1$ and $P$ is a
    simplex. Otherwise, the only possible $h$-vector of $\partial P$
    is $h(\partial P) = (1, 2, 2, \dots, 2, 2, 1)$, and from, for
    example, the lower bound theorem \cite{barnette_lbt}, we see that $P$ is a bipyramid
    over a simplex.

    We are left to examine what happens in this second case. For
    $1 \leq k \leq \frac{d}{2}$, we see
    \[
        2 \ = \ h_k(P) \ = \ \gInt_k(P) + \gInt_{d-k}(P) \ \leq \ 2\gInt_{d-k}(P)
        \ \leq \ 2\gInt_d(P) \ = \ 2\,,
    \]
    so $\gInt_k(P) = 2 - \gInt_{d-k}(P) = 1$ and
    $\gInt(P) = (0, 1, 1, \dots, 1, 1)$. With
    Proposition~\ref{prop:gInt_angle_proj}, we have
    \[
        0 \ = \ h_k(\partial P) - 2\gInt_k(P) \ = \ \sum_{R \in
          \Reg(P)} \alpha(R) g_k(\pi(R, P))\,.
    \]
    For any region $R \in \Reg(P)$, $g_k(\pi(R, P)) \geq 0$ by the
    characterization of the $h$-vectors of simplicial polytopes,
    Theorem~\ref{thm:hvectors}(3), since $\pi(R, P)$ is isomorphic to
    the boundary complex of a projection of $P$. It follows that
    $g(\pi(R, P)) = (1, 0, 0, \dots, 0, 0, -1)$ and that $\pi(R, P)$
    is the boundary of a simplex for all regions $R$ such that
    $\alpha(R) > 0$.
\end{proof}

Since the standard angle is positive for all full-dimensional
cones $C$, we have the following:
\begin{cor}
    Let $P$ be a simplicial $d$-polytope and let
    $\gInt(P) \defeq \gInt(\nu, P)$. If $P$ has a non-decreasing
    $\gInt$-vector, then $P$ is a simplex.
\end{cor}

In the remainder we want to examine the unimodality of
$\alpha$-symmetric polytopes. First we need the following lemma:

\begin{lem}
    Either a simplicial and $\alpha$-symmetric $d$-polytope $P$ is a
    simplex or $\gInt_{d-1}(P) \geq \gInt_d(P) = 1$.
\end{lem}
\begin{proof}
    Suppose that $P$ is not a simplex, then, $h_1(P) \geq 2$ and
    \[
        2\gInt_{d-1}(P) \ \geq \ \gInt_1(-P) + \gInt_{d-1}(P) \ = \
        \gInt_1(P) + \gInt_{d-1}(P) \ = \ h_1(P) \ \geq \ 2\,.\qedhere
    \]
\end{proof}

As a corollary, we can see that low dimensional, $\alpha$-symmetric
polytopes are unimodal:
\begin{prop}
    All simplicial and $\alpha$-symmetric $d$-polytopes, $d \leq 5$, have an unimodal
    $\gInt$-vector.
\end{prop}
\begin{proof}
    If $P$ is a simplex, then we are done by
    Lemma~\ref{lem:simplex}. Otherwise, for $d = 4, 5$, we know that
    $\gInt(P)$ is increasing until $\lceil \frac{d}{2} \rceil = d-2$
    and that $\gInt_{d-1}(P) \geq \gInt_d(P) = 1$. Thus either
    $\gInt_{d-2}(P) \geq \gInt_{d-1}(P)$ or
    $\gInt_{d-2}(P) < \gInt_{d-1}(P)$, but both possibilities give an
    unimodal $\gInt$-vector.
\end{proof}

Giving an counter example of an $\alpha$-symmetric simplicial polytope
with non-unimodal $\gInt$-vector is not as easy as in the
non-$\alpha$-symmetric case. The following example was found by
brute-force using a SAGE \cite{SAGE} script to generate random polytopes and
projections:

\begin{example}
    Let $P = \conv(V) \subseteq \R^6$ for $V$ being given as the
    following nine vertices.
    \[
        \arraycolsep=2pt
        \begin{array}{ *{21}{r}l }
          (&-48&  -8&  \phantom{-}16& -10&   \phantom{-0}6& -12&) \qquad (&-23&  -2&  -4&  6&   \phantom{-0}2&   \phantom{-0}8&) \qquad (&-20&  -8&   \phantom{-0}2&  \;-8&   \phantom{-0}9&  \phantom{-0}5&)\\
 (&-12&  -6&  -8&   4&   2&   2&) \qquad  (&  3&   \phantom{-0}0& \phantom{-0}2&  -2&   3&   1&) \qquad (& 22& -12&   2&  -9&   6&  -3&)\\
 (& 36&   0&  -8&   1&  -1&  -7&) \qquad  (& 50&  -9&   4& -10&   6&  -4&) \qquad (& 56&   6&   4&  -2&   0&  -2&)
        \end{array}
    \]
    Projecting $P$ orthogonally along the standard basis vector $e_1$ gives
    $Q \subseteq \R^5$. Then $\partial Q \cong \pi(R, P)$ for
    $R \in \Reg(P)$ with $e_1 \in \interior R$. Set
    $\alpha \defeq \frac{1}{2}(\omega_{e_1} + \omega_{-{e_1}})$ and
    $\gInt(P) = \gInt(\alpha, P)$. Note that $\alpha$ is even, thus
    $P$ is $\alpha$-symmetric. SAGE tells us that
    \begin{align*}
      h(\partial P) \ &= \ (1, 3, 4, 5, 4, 3, 1)\,,\\
      g(\partial Q) \ &= \ (1, 3, 0, 0, 0,-3,-1)
    \end{align*}
    and therefore
    \begin{align*}
        2\gInt(P) \ &= \ h(\partial P) - (\alpha(R) + \alpha(-R)) g(\pi(R, P))\\
                   \ &= \ (1, 3, 4, 5, 4, 3, 1) - (1, 3, 0, 0, 0, -3, -1) = (0, 0, 4, 5, 4, 6, 2)\,,
    \end{align*}
    which is not unimodal. If we flatten $P$ in direction $e_1$ with
    the following linear transformation for small enough
    $\eps > 0$
    \[
        C_\eps \ \defeq \ \begin{pmatrix}
            \eps&0&0&0&0&0\\
            0&1&0&0&0&0\\
            0&0&1&0&0&0\\
            0&0&0&1&0&0\\
            0&0&0&0&1&0\\
            0&0&0&0&0&1\\
        \end{pmatrix}\,,
    \]
    we see that all angles $\nu(F, P)$ for $F \in D(R, P)$ or
    $F \in B(R, P)$ tend to $\frac{1}{2}$, while those for
    $F \in \pi(R, P)$ tend to $0$. Thus we have
    \[
        \lim_{\eps \to 0} \gInt(\nu, C_\eps P) \ = \ (0, 0, 4, 5, 4, 6, 2)\,,
    \]
    so there are non-unimodal $\gInt$-vectors even for the standard
    cone angle.
\end{example}

While it is plausible to be false, it seems hard to find a
counterexample of a simplicial $d$-polytope with non-unimodal
$\gInt$-vector for an even cone angle and $d \geq 7$.

\bibliographystyle{siam} \bibliography{bibliography.bib}

\end{document}